\documentclass[11pt,reqno]{amsart}
\usepackage[utf8]{inputenc}
\usepackage{amsmath,amssymb}
\usepackage{wrapfig}
\usepackage{url}
\usepackage{mathtools}
\usepackage{graphicx}
\usepackage{stmaryrd}
\usepackage{amsthm}
\usepackage{xcolor}
\usepackage[colorlinks=true,linkcolor=blue,citecolor=blue]{hyperref}
\usepackage[shortlabels]{enumitem}
\usepackage{comment}
\usepackage{relsize}
\usepackage{dsfont}
\usepackage{stmaryrd}
\usepackage{geometry}
\usepackage{setspace}
 \usepackage{relsize}
 \usepackage{float}
\usepackage{mathrsfs} 
\usepackage{scrextend}
\usepackage{makecell}
\usepackage{orcidlink}
\usepackage{cancel}

\allowdisplaybreaks

\newtheorem{theorem}{Theorem}[section]
\newtheorem{corollary}[theorem]{Corollary}
\newtheorem{prop}[theorem]{Proposition}

\theoremstyle{definition}
\newtheorem{example}[theorem]{Example}

\theoremstyle{definition}
\newtheorem{obs}{Remark}
\theoremstyle{definition}
\newtheorem{definition}[theorem]{Definition}

\DeclareMathOperator{\R}{\mathbb{R}}

\DeclareMathOperator{\N}{\mathbb{N}}

\newcommand{\p}{p(\cdot)}

\newcommand{\q}{q(\cdot)}
\newcommand{\Lp}{L^{\p}(\R)}
\newcommand{\BV}{\mathcal{B}V}
\usepackage{mathtools}
\newcommand{\defeq}{\vcentcolon=}

\newcommand{\LH}{\operatorname{LH}_\infty^1}
\date{\today}

\title[The Fourier transform in variable exponent Lebesgue spaces]{The Fourier transform in variable exponent Lebesgue spaces}
\author{André Kowacs \orcidlink{https://orcid.org/0000-0001-8784-3147}}
 \address{Universidade Federal do Paran\'{a}, 
	Departamento de Matem\'{a}tica,
	C.P.19096, CEP 81531-990, Curitiba, Brazil}
\email{andrekowacs@gmail.com}
   \endgraf

 \author{Wagner Augusto Almeida de Moraes \orcidlink{https://orcid.org/0000-0002-5624-7374}}
 \address{Universidade Federal do Paran\'{a}, 
	Departamento de Matem\'{a}tica,
	C.P.19096, CEP 81531-990, Curitiba, Brazil}
\email{wagnermoraes@ufpr.br}



\begin{document}
\begin{abstract}
   In this work we define a Fourier transform for each $f\in L^{\p}(\R)$, for a large class of exponent functions $\p$, as the distributional derivative of a Hölder continuous function. A norm is defined in the space of such Fourier transforms so that it is isometrically isomorphic to $\Lp$. We also prove several properties of this Fourier transform, such as inversion in norm and an exchange theorem.
\end{abstract}
\maketitle
\keywords

\section{Introduction}

The variable exponent Lebesgue spaces have been an area of intensive study in the last 30 years, with historical roots tracing back to the work of Orlicz in the 1930s on modular spaces \cite{Orlicz1932}, which laid the groundwork for generalizations beyond constant exponents. The modern formulation of these spaces, particularly the variable exponent Lebesgue spaces $L^{p(\cdot)}(\Omega)$, where $\Omega\subset \R^n$, was significantly advanced by Kováčik and Rákosník \cite{KovacikRakosnik1991} and Cruz-Uribe and Fiorenza \cite{Cruz-Uribe}. These spaces can be seen as a natural generalization of the classical Lebesgue spaces. Roughly speaking, the usual constant parameter $p\in [1,\infty]$ in the definition of $L^p({\Omega})$ is replaced by a measurable function $p(\cdot):\mathbb{R}^n\to [1,\infty]$. Although this generalization retains many properties of the classical Lebesgue spaces, the variable exponent Lebesgue spaces also display some very different phenomena depending on the exponent function $p(\cdot)$. Therefore extending results from $L^p$ to $L^{p(\cdot)}$ is in general a highly non-trivial problem and involves different techniques, a fact which has served as motivation for several papers in the last decades (see for instance  \cite{Chamorro, CUHM16, Diening2004, Diening2004b, Sam98}, and references therein).

With this motivation, in the present work we extend several results from Talvila \cite{Talvila2025}. In that paper, the author introduced an alternative definition of the Fourier transform for functions $f \in L^p(\mathbb{R})$, $1\leq p<\infty$, establishing that the classical Fourier transform (seen as a tempered distribution) coincides with the distributional derivative of a Hölder continuous function. The work also develops additional properties of the Fourier transform $\widehat{f}$ via the theory of the continuous primitive integral, including exchange theorems and norm inversion results.

It is worth mentioning that when restricted to constant exponents, our results coincide precisely with the ones obtained in \cite{Talvila2025}. Moreover, while many papers on variable exponent Lebesgue spaces restrict to the case where the exponent function $p(\cdot)$ assumes values in $[1,\infty)$, in this paper we consider the full range of values $[1,\infty]$, which allows us to obtain more general results. Notably, when the essential supremum of $p(\cdot)$ is infinite, the analysis diverges significantly from the constant exponent case (see, for example, \cite{Hasto2009}). This is due to some characteristics of the structure of $L^\infty(\mathbb{R})$, such as the lack of certain approximation properties and the failure of density results, requiring distinct techniques compared to finite exponent scenarios.  

The paper is structured as follows: Section \ref{section_preliminaries} reviews the fundamental theory of variable exponent Lebesgue spaces,  and fix the notation that will be used throughout the paper. Section \ref{sec-prop-psi} introduces a key integral kernel function for Fourier analysis, establishing its regularity under diverse exponent conditions through Hölder and Lipschitz continuity estimates. In Section \ref{sec-fourier-transf} we define two Banach spaces, while describing the Fourier transform rigorously as a distributional derivative, and show that these spaces are isometrically isomorphic to $L^{\p}(\mathbb{R})$. Section \ref{sec-prop-fourier} derives operator calculus for this transform, establishing its consistency with tempered distributions and invariance under translation operations. Finally, Section \ref{sec-integration-fourier} develops an integration framework for Fourier transforms, proving both exchange theorems and norm-convergent inversion formulas that extend classical Fourier analysis.

\section{Preliminaries}\label{section_preliminaries}

    For the main theory and results listed in this section, we refer to \cite{Cruz-Uribe}. Throughout this paper, we fix the usual Lebesgue measure on the real line.   Given a set $\Omega\subset \R^n$, let $\mathcal{P}(\Omega)$ denote the set of all measurable functions $p:\Omega\to[1,\infty]$. These are called exponent functions, and to distinguish between variable and constant exponents, we will denote them by $\p$. We also fix the following notations:
    \begin{equation*}
        \Omega_1^{\p}=\{x\in\Omega : p(x)=1\},
        \end{equation*}
        \begin{equation*}
                    \Omega_*^{\p}=\{x\in\Omega : 1<p(x)<\infty\},
        \end{equation*}
\begin{equation*}
        \Omega_\infty^{\p}=\{x\in\Omega : p(x)=\infty\},
    \end{equation*}
        \begin{equation*}
        p_{-}=\underset{{x\in \Omega}}{\text{ess inf}}\{p(x)\},
    \end{equation*}
    \begin{equation*}
        p_{+}=\underset{{x\in \Omega}}{\text{ess sup}}\{p(x)\}.
    \end{equation*}
    Given $\p\in\mathcal{P}(\Omega)$, its conjugate exponent $\q\in\mathcal{P}(\Omega)$ is defined by 
    \begin{equation*}
        \frac{1}{p(x)}+\frac{1}{q(x)}=1,\quad x\in\Omega.
    \end{equation*}
    In this paper, $\q$ will always denote the respective  conjugate exponent of any given $\p$.

Note that as a consequence of these definitions, the following symmetric identities always hold:
\begin{equation*}
    \frac{1}{p_{+}}+\frac{1}{q_{-}}=1,
\end{equation*}
\begin{equation*}
     \frac{1}{p_{-}}+\frac{1}{q_{+}}=1.
\end{equation*}
    
A key condition for the exponent functions which ensures important properties for the $L^{\p}(\Omega)$ spaces (see \cite{heat_equation,log-2019,key-estimate}), is the so called log-Hölder continuity. An exponent $\p\in\mathcal{P}(\Omega)$ is locally log-Hölder continuous if  there exists a constant $c>0$ such that
\begin{equation*}
    |p(x)-p(y)|\leq \frac{c}{\log(e+1/|x-y|)},
\end{equation*}
for every $x,y\in\Omega$. Also, $\p$ is log-Hölder continuous at infinity if there exist constants $C,p_{\infty}>0$ such that
\begin{equation}\label{ineq_log-holder}
    |p(x)-p_{\infty}|\leq \frac{C}{\log(e+|x|)},
\end{equation}
for every $x\in\Omega$. If $\frac{1}{\p}$ is log-Hölder continuous at infinity, we say that it satisfies the log-Hölder decay condition.  In this paper we will assume only that \eqref{ineq_log-holder} holds for almost every $x\in\R$.

Given $\p \in \mathcal{P}(\Omega)$ and  $\Omega'\subset\Omega$ the modular $\rho_{\p,\Omega'}$ associated with $\p$ denotes the functional defined on Lebesgue measurable functions given by
\begin{equation*}
    \rho_{\p,\Omega'}(f)=\int_{\Omega'\backslash\Omega_\infty^{\p}}|f(x)|^{p(x)}dx+\|f\|_{L^{\infty}(\Omega'\cap \Omega_\infty^{\p})}\in [0,\infty].
\end{equation*}
Often during this work $\Omega'=\Omega$, and we abbreviate the modular by $\rho_{\p}(f)$.

The variable exponent Lebesgue space $L^{\p}(\Omega)$ is then defined as the set of measurable functions $f$ such that $\rho_{\p}(f/\lambda)<\infty$, for some $\lambda>0$, with the usual convention that two functions are ``equal" if they coincide almost everywhere.

The norm in $L^{\p}$ is then defined as follows. Given $\Omega\subset\R^n$ and $\p\in\mathcal{P}(\Omega)$, define
\begin{equation*}
    \|f\|_{L^{\p}(\Omega)}=\inf\{\lambda>0:\rho_{\p,\Omega}(f/\lambda)\leq 1\},
\end{equation*}
for every Lebesgue measurable function $f$, with the standard convention that $\inf\varnothing=\infty$. When the set $\Omega$ is implicit, we will often abbreviate the norm by $\|f\|_{\p}$.

\begin{obs}\label{remark_no_lambda}
    Notice that $\|f\|_{L^{\p}(\Omega)}<\infty$ for every $f\in L^{\p}(\Omega)$. In particular, if $\rho_{\p}(f)<\infty$, then $\|f\|_{L^{\p}(\Omega)}<\infty$ as well.
\end{obs}

The variable exponent Lebesgue spaces extend the usual Lebesgue spaces, and share similar properties. In particular, the variable exponent Lebesgue space norm satisfies a generalization of Hölder's inequality, presented below. First, we introduce a notation to simplify several statements throughout this paper.

\begin{definition}
    For a Lebesgue measurable set $E\subset \R^n$, denote by $|E|$ its Lebesgue measure. Define $\delta(E)=\|\chi_{E}\|_{L^\infty(\R^n)}$, that is, $\delta(E)=0$ if $|E|=0$ and $\delta(E)=1$, otherwise.
\end{definition}

\begin{theorem}[Generalized Hölder's Inequality]
    Given $\Omega\subset\R^n$ and $\p\in\mathcal{P}(\Omega)$, for all $f\in L^{\p}(\Omega)$ and $g\in L^{\q}(\Omega)$, $fg\in L^1(\Omega)$ and 
    \begin{equation*}
        \int_\Omega|f(x)g(x)|dx\leq K_{\p}\|f\|_{\p}\|g\|_{\q},
    \end{equation*}
    where 
    \begin{equation*}
        K_{\p}=\left(\frac{1}{p_{-}}-\frac{1}{p_{+}}+1\right)\delta\left(\Omega_*^{\p}\right)+\delta\left(\Omega_\infty^{\p}\right)+\delta\left(\Omega_1^{\p}\right),
    \end{equation*}
    and $\q\in\mathcal{P}(\Omega)$ is the conjugate exponent of $\p$. 
\end{theorem}

\section{Properties of \texorpdfstring{$\Psi_f$}{}}\label{sec-prop-psi}
In this section we define, for each $f\in\Lp$, an auxiliary function $\Psi_f:\R\to\R$ which will be essential in defining the Fourier transform of $f$. We also prove several properties of this function $\Psi_f$, based on properties of the exponent $\p$.

\begin{definition}
    Let $f\in L^{\p}(\R)$, for some $\p\in\mathcal{P}(\R)$. Define
    \begin{equation*}
        \Psi_f(s)\defeq\int_{-\infty}^{\infty}\left(\frac{1-e^{-ist}}{it}\right)f(t)dt,
    \end{equation*}
    for $s\in\R\backslash\{0\}$, whenever the integral is well defined, and $\Psi_f(0)=0$. For $1<q<\infty$, set
    \begin{equation}
      C_q= 4^{1/q}\left(\int_0^{\infty}\left|\frac{\sin(y)}{y}\right|^{q}dy\right)^{1/q}<\infty,
    \end{equation}
    and $C_\infty=1$. Also let 
    \begin{equation*}
            u_s(t)\defeq\frac{1-e^{ist}}{it}, 
    \end{equation*}
for every $s,t\in\R\backslash\{0\}$.
\end{definition}

\begin{prop}\label{prop_bounded_-_+}
    Suppose that $p(\cdot)\in \mathcal{P}(\R)$  satisfies $1<p_{-}, p_{+}<\infty$. Then for any $f\in L^{p(\cdot)}(\R)$ we have that
    \begin{equation}\label{ineq_Psi_bounded}
        |\Psi_f(s)|\leq K_{p(\cdot)}\|f\|_{p(\cdot)}\max\left\{\left(C_{q_{-}}\right)^{q_{-}/q_{+}}|s|^{\frac{1}{p_{-}}},C_{q_{-}}|s|^{\frac{1}{p_{+}}}\right\},
    \end{equation}
    for every $s\in\R$.
\end{prop}

\begin{proof}
 First, notice that due to the generalized Hölder's inequality we have that
    \begin{align*}
        |\Psi_f(s)|&\leq \int_{-\infty}^{\infty}|u_s(t)||f(t)|dt\\
        &\leq K_{p(\cdot)}\|u_s\|_{q(\cdot)}\|f\|_{p(\cdot)},
    \end{align*}
     Therefore, to obtain the estimate in the claim, it is enough to estimate $\|u_s\|_{q(\cdot)}$. For that, we estimate the modular $\rho_{q(\cdot)}(u_s/\lambda)$. Note that since $\left|\R_\infty^{q(\cdot)}\right|=\left|\R_1^{p(\cdot)}\right|=0$ we have that
    \begin{align*}
        \rho_{q(\cdot)}(u_s/\lambda)&\leq \int_{-\infty}^{\infty}\left|\frac{1-e^{-ist}}{it\lambda}\right|^{q(t)}dt
        \\
        &=\int_{-\infty}^{\infty}\left|\frac{\sin(st/2)}{t/2}\right|^{q(t)}\lambda^{-q(t)}dt.
    \end{align*}
    Since the inequality above is symmetric in $s$ and \eqref{ineq_Psi_bounded} is clearly true for $s=0$, we may assume that $s>0$. Performing the change of variables $st=x$, we obtain
    \begin{align*}
        \rho_{q(\cdot)}(u_s/\lambda)&\leq\int_{-\infty}^{\infty}\left|\frac{\sin(x/2)s}{x/2}\right|^{q(x/s)}\lambda^{-q(x/s)}dx\frac{1}{s}
        \\
        &=\int_{-\infty}^{\infty}\left|\frac{\sin(x/2)}{x/2}\right|^{q(x/s)}\left(\frac{s}{\lambda}\right)^{q(x/s)}dx\frac{1}{s}.
    \end{align*}
     Consider $0<\lambda\leq s$.  Since $|\sin(x)/x|\leq 1$ for all $x\in\R$, from the inequality above we have that
     \begin{align*}
        \rho_{q(\cdot)}(u_s/\lambda)&\leq 
        \int_{-\infty}^{\infty}\left|\frac{\sin(x/2)}{x/2}\right|^{q_{-}}\left(\frac{s}{\lambda}\right)^{q_{+}}dx\frac{1}{s}
       \\
&=4\int_0^{\infty}\left|\frac{\sin(y)}{y}\right|^{q_{-}}dy\ s^{(q_{+}-1)}\lambda^{-q_{+}}
       \\
       &= (C_{q_{-}})^{q_{-}}s^{(q_{+}-1)}\lambda^{-q_{+}},
    \end{align*}
     where in the second line we performed the change of variables $y=x/2$ and used the fact that $|\sin(x)/x|$ is even. From the inequality above, we conclude that if 
   \begin{equation}\label{lambda_p_->1}
\lambda\geq  (C_{q_{-}})^{\frac{q_{-}}{q_{+}}}s^{\frac{1}{p_{-}}}
   \end{equation}
   and 
   $\lambda\leq s$,
   we have that $\rho_{q(\cdot)}(u_s/\lambda)\leq 1$.
   However, both inequalities can be satisfied for some $\lambda\in \R$ only if $s\geq  (C_{q_{-}})^{\frac{q_{-}}{q_{+}}}s^{\frac{1}{p_{-}}}$, or equivalently, $s\geq (C_{q_{-}})^{q_{-}}$. Hence, $ \|u_s\|_{q(\cdot)}\leq(C_{q_{-}})^{\frac{q_{-}}{q_{+}}}s^{\frac{1}{p_{-}}}$ if $s\geq (C_{q_{-}})^{q_{-}}$.
Suppose now $s\leq (C_{q_{-}})^{q_{-}}$ and consider $\lambda>s$. Similarly as above, we have that
     \begin{align*}
        \rho_{q(\cdot)}(u_s/\lambda)&\leq \int_{-\infty}^{\infty}\left|\frac{\sin(x/2)}{x/2}\right|^{q(x/s)}\left(\frac{s}{\lambda}\right)^{q(x/s)}dx\frac{1}{s}
        \\
       &\leq \int_{-\infty}^{\infty}\left|\frac{\sin(x/2)}{x/2}\right|^{q_{-}}\left(\frac{s}{\lambda}\right)^{q_{-}}dx\frac{1}{s}
       \\
       &=4\int_0^{\infty}\left|\frac{\sin(y)}{y}\right|^{q_{-}}dy\ s^{(q_{-}-1)}\lambda^{-q_{-}}
       \\
       &\leq (C_{q_{-}})^{q_{-}}s^{(q_{-}-1)}\lambda^{-q_{-}}.
    \end{align*}
   From the inequality and considerations  above, we conclude that if 
   \begin{equation*}
\lambda> \max\left\{s,C_{q_{-}}s^{\frac{1}{p_{+}}}\right\},
   \end{equation*}
   we have that $\rho_{q(\cdot)}(u_s/\lambda)\leq 1$.
 Therefore $\|u_s\|_{q(\cdot)}\leq  \max\left\{s,C_{q_{-}}s^{\frac{1}{p_{+}}}\right\}$. But since we are assuming $0<s\leq (C_{q_{-}})^{q_{-}}$, we have that $s\leq C_{q_{-}}s^{\frac{1}{p_{+}}}$, or equivalently, $s^{1-\frac{1}{p_{+}}}\leq C_{q_{-}}$. Indeed, note that $s\mapsto s^{1-\frac{1}{p_{+}}}=s^{\frac{1}{q_{-}}}$ is increasing for $s>0$ and assumes the value $C_{q_{-}}$ when $ s=(C_{q_{-}})^{q_{-}}$.
 Therefore 
 we conclude that $\|u_s\|_{q(\cdot)}\leq C_{q_{-}}s^{\frac{1}{p_{+}}}$ if $s\leq (C_{q_{-}})^{q_{-}}$. Finally, using both estimates we conclude that   $\|u_s\|_{q(\cdot)}\leq  \max\left\{(C_{q_{-}})^{\frac{q_{-}}{q_{+}}}s^{\frac{1}{p_{-}}},C_{q_{-}}s^{\frac{1}{p_{+}}}\right\}$, which finishes the proof.
\end{proof}

\begin{obs}
    We note that Proposition \ref{prop_bounded_-_+} extends the analogous result for the constant coefficient case, that is, $f\in L^p(\R)$, where $1<p<\infty$, as proven in Theorem 2.1 (a) in \cite{Talvila2025}, recovering the same constants when $\p$ does not depend on $x\in\R$.
\end{obs}

\begin{prop}\label{prop_bounded}
    Suppose that $p(\cdot)\in \mathcal{P}(\R)$ satisfies $p_{+}<\infty$. Then for any $f\in L^{p(\cdot)}(\R)$ we have that
    \begin{equation}\label{ineq_Psi_general}
        |\Psi_f(s)|\leq \left\{\begin{alignedat}{3}
            &\|f\|_{\p}|s|, &&p_{+}=1,\\
            &K_{\p}\|f\|_{\p}\max\left\{|s|,C_{q_{-}}|s|^{\frac{1}{p_{+}}}\right\}, &&p_{+}>p_{-}=1,\\
&K_{p(\cdot)}\|f\|_{p(\cdot)}\max\left\{(C_{q_{-}})^{\frac{q_{-}}{q_{+}}}|s|^{\frac{1}{p_{-}}},C_{q_{-}}|s|^{\frac{1}{p_{+}}}\right\}, \quad&&p_{-}>1.
       \end{alignedat} \right.
    \end{equation}
    for every $s\in\R$. Alternatively
\begin{equation}\label{ineq_Psi_general_2}
        |\Psi_f(s)|\leq 
         \left\{\begin{alignedat}{3}
            &\|f\|_{\p}|s|, &&p_{+}=1, \\
            &K_{p(\cdot)}\|f\|_{p(\cdot)}\max\left\{(C_{q_{-}})^{\frac{q_{-}}{q_{+}}}|s|^{\frac{1}{p_{-}}},C_{q_{-}}|s|^{\frac{1}{p_{+}}}\right\}, \quad&&\text{otherwise}.
        \end{alignedat}\right.
    \end{equation}
    with the convention that $q_{-}/\infty=0$. Furthermore, if we fix that $1^{\frac{\infty}{\infty}}=1$, we can consider only the second formula in \eqref{ineq_Psi_general_2}.
\end{prop}
\begin{proof}
First, notice that if $p_{+}=1$, then $p\equiv 1$ and there is nothing to prove.
 Next, notice that the case  $p_{-}>1$ was already proved in Proposition  \ref{prop_bounded_-_+}. The only case remaining is for $p_{+}>p_{-}=1$. In this case, for $\lambda'>0$  due to the generalized Hölder's inequality we have that
 \begin{align}
     |{\Psi_f(s)}|&\leq{\lambda'}\delta\left(\R_1^{\p}\right)\int_{\R^{\p}_1}|u_s(t)|\frac{|f(t)|}{\lambda'}dt+\int_{\R\backslash\R^{\p}_1}|u_s(t)||f(t)|dt\notag\\
     &\leq \lambda'\delta\left(\R_1^{\p}\right)\int_{\R^{\p}_1}\frac{|f(t)|}{\lambda'}dt|s|+K_{L^{\p}({\R\backslash\R^{\p}_1)}}\|u_s\|_{L^{q(\cdot)}(\R\backslash\R_{\infty}^{q(\cdot)})}\|f\|_{L^{\p}({\R\backslash\R^{\p}_1)}}\notag\\
     &\leq \lambda'\delta\left(\R_1^{\p}\right)\rho_{\p,\R^{\p}_1}(f/\lambda')|s|+\left(\frac{1}{p_{-}}-\frac{1}{p_{+}}+1\right)\|u_s\|_{L^{q(\cdot)}(\R\backslash\R_{\infty}^{q(\cdot)})}\|f\|_{\p},\label{ineq_bounded}
 \end{align}
    where in the last line we used the fact that 
    \begin{equation*}
        K_{L^{\p}({\R\backslash\R^{\p}_1)}}=\left(\frac{1}{p_{-}}-\frac{1}{p_{+}}+1\right),
    \end{equation*}
    because $\left(\R\backslash\R^{\p}_1\right)^{\p}_1=\emptyset$, $|\R^{\p}_{\infty}|=0$ and $\p\not\equiv1\implies |\R^{\p}_*|\neq 0$.
    Since $f\in \Lp$, taking  $\lambda'=\|f\|_{\p}$ by \cite[Proposition 2.21]{Cruz-Uribe} we have that $\rho_{\p,\R^{\p}_1}(f/\lambda')\leq \rho_{\p}(f/\lambda')\leq1$. Next we estimate $\|u_s\|_{L^{q(\cdot)}(\R\backslash\R_{\infty}^{q(\cdot)})}$. 
    Note that since $q(x)<\infty$ on $\R\backslash\R_{\infty}^{q(\cdot)}$ we have that
    \begin{align*}
        \rho_{q(\cdot),\R\backslash\R_{\infty}^{q(\cdot)}}(u_s/\lambda)&\leq \int_{\R\backslash\R_{\infty}^{q(\cdot)}}\left|\frac{1-e^{-ist}}{it\lambda}\right|^{q(t)}dt
        \\
        &=\int_{\R\backslash\R_{\infty}^{q(\cdot)}}\left|\frac{\sin(st/2)}{t/2}\right|^{q(t)}\lambda^{-q(t)}dt
    \end{align*}
    As the inequality above is symmetric in $s$ and \eqref{ineq_Psi_bounded} is clearly true for $s=0$, we may assume that $s>0$. Performing the change of variables $st=x$, we obtain
    \begin{align*}
       \rho_{q(\cdot),\R\backslash\R_{\infty}^{q(\cdot)}}(u_s/\lambda)&\int_{\frac{x}{s}\in\R\backslash\R_{\infty}^{q(\cdot)}}\left|\frac{\sin(x/2)}{x/2}\right|^{q(x/s)}\left(\frac{s}{\lambda}\right)^{q(x/s)}dx\frac{1}{s}
    \end{align*}
  Since $p_{+}<\infty$ implies $q_{-}>1$ and since $|\sin(x)/x|\leq 1$ for all $x\in\R$, by considering $\lambda\geq s$ the inequality above implies that
     \begin{align*}
\rho_{q(\cdot),\R\backslash\R_{\infty}^{q(\cdot)}}(u_s/\lambda)&
       \leq \int_{-\infty}^{\infty}\left|\frac{\sin(x/2)}{x/2}\right|^{q_{-}}\left(\frac{s}{\lambda}\right)^{q_{-}}dx\frac{1}{s}
       \\
       &=4\int_0^{\infty}\left|\frac{\sin(y)}{y}\right|^{q_{-}}dy\ s^{(q_{-}-1)}\lambda^{-q_{-}}
       \\
       &\leq (C_{q_{-}})^{q_{-}}s^{(q_{-}-1)}\lambda^{-q_{-}}.
    \end{align*}
     Hence, we conclude that if 
   \begin{equation*}
\lambda\geq  \max\left\{C_{q_{-}}s^{\frac{1}{p_{+}}},s\right\}
   \end{equation*}
     we have that $\rho_{q(\cdot),\R\backslash\R_{\infty}^{q(\cdot)}}(u_s/\lambda)\leq 1$. Therefore $ \|u_s\|_{{L^{q(\cdot)}(\R\backslash\R_{\infty}^{q(\cdot)})}}\leq  \max\left\{s,C_{q_{-}}s^{\frac{1}{p_{+}}}\right\}$. Together with \eqref{ineq_bounded} and the previous considerations, this implies that
   \begin{align*}
       |\Psi_f(s)|&\leq \delta\left(\R_1^{\p}\right)\|f\|_{\p}|s|+\left(\frac{1}{p_{-}}-\frac{1}{p_{+}}+1\right)\|f\|_{p(\cdot)}\max\left\{|s|,C_{q_{-}}|s|^{\frac{1}{p_{+}}}\right\}\\
       &\leq \|f\|_{\p}\left(\delta\left(\R_1^{\p}\right)|s|+\left(\frac{1}{p_{-}}-\frac{1}{p_{+}}+1\right)\max\left\{|s|,C_{q_{-}}|s|^{\frac{1}{p_{+}}}\right\}\right).
   \end{align*}
   As in the proof of Proposition \ref{prop_bounded_-_+}, if $|s|\leq (C_{q_{-}})^{q_{-}}$ then $\max\left\{|s|,C_{q_{-}}|s|^{\frac{1}{p_{+}}}\right\}=C_{q_{-}}|s|^{\frac{1}{p_{+}}}$, otherwise
    $\max\left\{|s|,C_{q_{-}}|s|^{\frac{1}{p_{+}}}\right\}=|s|$. Hence, for $|s|\leq (C_{q_{-}})^{q_{-}}$
 we have that
 \begin{align*}
     |\Psi_f(s)| &\leq \|f\|_{\p}\left(\delta\left(\R_1^{\p}\right)|s|+\left(\frac{1}{p_{-}}-\frac{1}{p_{+}}+1\right)C_{q_{-}}|s|^{\frac{1}{p_{+}}}\right)\\
     &= \|f\|_{\p}\left(\delta\left(\R_1^{\p}\right)|s|^{\frac{1}{q_{-}}}+\left(\frac{1}{p_{-}}-\frac{1}{p_{+}}+1\right)C_{q_{-}}\right)|s|^{\frac{1}{p_{+}}}\\
     &\leq \|f\|_{\p}\left(\delta\left(\R_1^{\p}\right)C_{q_{-}}+\left(\frac{1}{p_{-}}-\frac{1}{p_{+}}+1\right)C_{q_{-}}\right)|s|^{\frac{1}{p_{+}}}\\
     &=\|f\|_{\p}K_{\p}C_{q_{-}}|s|^{\frac{1}{p_{+}}}.
 \end{align*}
 On the other hand, for $|s|>(C_{q_{-}})^{q_{-}}$ we have that
 \begin{align*}
      |\Psi_f(s)|&\leq \|f\|_{\p}\left(\delta\left(\R_1^{\p}\right)|s|+\left(\frac{1}{p_{-}}-\frac{1}{p_{+}}+1\right)|s|\right)= \|f\|_{\p}K_{\p}|s|.
 \end{align*}
 Therefore
 \begin{align*}
       |\Psi_f(s)|\leq\|f\|_{\p}K_{\p}\max\left\{|s|,C_{q_{-}}|s|^{\frac{1}{p_{+}}}\right\}.
   \end{align*}
   as claimed, finishing the proof.
\end{proof}

In what follows, we denote by $W:\R_+\to \R_+$ the main branch of the Lambert W function restricted to the positive real numbers, that is, $W(x)=y$ is the only real solution to the equation $ye^{y}=x$, for every $x>0$.

\begin{prop}\label{coro_rapid_decay}
     Suppose that $1/p(\cdot)\in \mathcal{P}(\R)$ is log-Hölder continuous at infinity with constants $p_{\infty}=1$ and $C>0$. Then for any $f\in L^{p(\cdot)}(\R)$ we have that
    \begin{equation}\label{ineq_Psi_log-holder}
        |\Psi_f(s)|\leq \|f\|_{\p}\left(\delta\left(\R^{\p}_1\right)+\left(2-\frac{1}{p_{+}}\right)e^{C(W(2)+1)}\right)|s|,
    \end{equation}
    for every $s\in\R$.
\end{prop}
\begin{proof}
Notice that for $\lambda'>0$  due to the generalized Hölder's inequality we have that
 \begin{align*}
     |{\Psi_f(s)}|&\leq\delta\left(\R^{\p}_1\right){\lambda'}\int_{\R^{\p}_1}|u_s(t)|\frac{|f(t)|}{\lambda'}dt+\int_{\R\backslash\R^{\p}_1}|u_s(t)||f(t)|dt\\
     &\leq \delta\left(\R^{\p}_1\right)\lambda'\int_{\R^{\p}_1}\frac{|f(t)|}{\lambda'}dt|s|+K_{L^{\p}({\R\backslash\R^{\p}_1)}}\|u_s\|_{L^{q(\cdot)}(\R\backslash\R_{\infty}^{q(\cdot)})}\|f\|_{L^{\p}({\R\backslash\R^{\p}_1)}}\\
     &\leq \delta\left(\R^{\p}_1\right)\lambda'\rho_{\p,\R^{\p}_1}(f/\lambda')|s|+\left(1+\frac{1}{p_{-}}-\frac{1}{p_{+}}+\delta\left(\R_{\infty}^{\p}\right)\right)\|u_s\|_{L^{q(\cdot)}(\R\backslash\R_{\infty}^{q(\cdot)})}\|f\|_{\p},
 \end{align*}
    where in the last line we used the fact that 
    \begin{equation*}
        K_{L^{\p}({\R\backslash\R^{\p}_1)}}\leq1+\frac{1}{p_{-}}-\frac{1}{p_{+}}+\delta\left(\R_{\infty}^{\p}\right),
    \end{equation*}
   Once again, since $f\in \Lp$, taking  $\lambda'=\|f\|_{\p}$ by \cite[Proposition 2.21]{Cruz-Uribe} we have that $\rho_{\p,\R^{\p}_1}(f/\lambda')\leq \rho_{\p}(f/\lambda')\leq1$.  Next we estimate $\|u_s\|_{L^{q(\cdot)}(\R\backslash\R_{\infty}^{q(\cdot)})}$.
    Note that since $q(x)<\infty$ on $\R\backslash\R_{\infty}^{q(\cdot)}$ we have that
    \begin{align*}
        \rho_{q(\cdot),\R\backslash\R_{\infty}^{q(\cdot)}}(u_s/\lambda)&=\int_{\R\backslash\R_{\infty}^{q(\cdot)}}\left|\frac{1-e^{-ist}}{it\lambda}\right|^{q(t)}dt\\
        &=\int_{\R\backslash\R_{\infty}^{q(\cdot)}}\left|\frac{\sin(st/2)}{t/2}\right|^{q(t)}\lambda^{-q(t)}dt.
    \end{align*}
    Since the inequality above is symmetric in $s$ and \eqref{ineq_Psi_log-holder} is trivially true for $s=0$, we may assume that $s>0$. As $\left|\frac{\sin(st/2)}{t/2}\right|\leq s$ for all $t\in\R$, we have that
    \begin{align}
        \rho_{q(\cdot),\R\backslash\R_{\infty}^{q(\cdot)}}(u_s/\lambda)&=\int_{\R\backslash\R_{\infty}^{q(\cdot)}}\left|\frac{\sin(st/2)}{t/2}\right|^{q(t)}\lambda^{-q(t)}dt\notag\\
        &\leq \int_{\R\backslash\R_{\infty}^{q(\cdot)}}\left(\frac{s}{\lambda}\right)^{q(t)}dt.\label{ineq_general_tau_q}
    \end{align}
Notice that since $1/\p$ is log-Hölder continuous at infinity with constants $p_\infty=1$ and $C>0$ we have that
\begin{align*}
    \frac{1}{q(x)}=1-\frac{1}{p(x)}\leq \frac{C}{\log(e+|x|)},
\end{align*}
for almost every $x\in\R$, hence
\begin{equation*}
    q(x)\geq C^{-1}\log(e+|x|),
\end{equation*}
for almost every $x\in\R$.
Considering $\lambda> s$ by the inequality above we have that
    \begin{align*}
        \rho_{q(\cdot),\R\backslash\R_{\infty}^{q(\cdot)}}(u_s/\lambda)&\leq \int_{\R\backslash\R_{\infty}^{q(\cdot)}}\left(\frac{s}{\lambda}\right)^{C^{-1}\log(e+|t|)}dt\\
        &\leq\int_{-\infty}^{\infty}(e+|t|)^{-C^{-1}\log(\lambda/s)}dt\\
        &=2\int_e^{\infty}t^{-C^{-1}\log(\lambda/s)}dt.
    \end{align*}
    If we also consider $\lambda>se^{C}>s$, so that $C^{-1}\log(\lambda/s)>1$, we have that
   \begin{align*}
         \rho_{q(\cdot),\R\backslash\R_{\infty}^{q(\cdot)}}(u_s/\lambda)&\leq2\left(\frac{t^{-C^{-1}\log(\lambda/s)+1}}{-C^{-1}\log(\lambda/s)+1}\right)_{t=e}^{t=\infty}\\
         &=2\left(\frac{e^{-(C^{-1}\log(\lambda/s)-1)}}{C^{-1}\log(\lambda/s)-1}\right).
   \end{align*}
   Notice that for $x=(C^{-1}\log(\lambda/s)-1)>0$, we have that
   \begin{equation*}
       2\left(\frac{e^{-(C^{-1}\log(\lambda/s)-1)}}{C^{-1}\log(\lambda/s)-1}\right)=1\iff\frac{1}{xe^{x}}=\frac{1}{2}\iff xe^{x}=2\iff x=W(2).
   \end{equation*}
   Hence, we conclude that if 
   \begin{equation*}
C^{-1}\log(\lambda/s)-1\geq W(2)\iff \lambda\geq se^{C(W(2)+1)}
   \end{equation*}
   and $\lambda> se^{C}$, we have that $  \rho_{q(\cdot),\R\backslash\R_{\infty}^{q(\cdot)}}(u_s/\lambda)\leq 1$. But since $W(2)>0$, it is enough that $\lambda \geq se^{C(W(2)+1)}$.
   Therefore  $\|u_s\|_{{L^{q(\cdot)}(\R\backslash\R_{\infty}^{q(\cdot)})}}\leq se^{C(W(2)+1)}=|s|e^{C(W(2)+1)}$. Finally, since $p_{-}=1$, we have that $  K_{L^{\p}({\R\backslash\R^{\p}_1)}}\leq2-\frac{1}{p_{+}}+\delta\left(\R_{\infty}^{\p}\right)$, and so
   \begin{align*}
       |\Psi_f(s)|&\leq \delta\left(\R^{\p}_1\right)\|f\|_{\p}|s|+\left(2-\frac{1}{p_{+}}\right)\|f\|_{p(\cdot)}e^{C(W(2)+1)}|s|\\
       &= \|f\|_{\p}\left(\delta\left(\R^{\p}_1\right)+\left(2-\frac{1}{p_{+}}\right)e^{C(W(2)+1)}\right)|s|,
   \end{align*}
   as claimed.
\end{proof}
In view of Proposition \ref{coro_rapid_decay}, we establish the following notation.
\begin{definition}
    Let $\p\in\mathcal{P}(\R)$. We will say that $\p$ is $\LH$ if $\p$ is log-Hölder continuous at infinity with $p_\infty=1$.
\end{definition}

Next we provide an equivalent characterization of the log-Hölder decay condition with $p_\infty=1$ which can be easier to verify and useful to present examples.

\begin{prop}\label{lemma_tau}
    Let $\p\in\mathcal{P}(\R)$. Then $1/\p$ is $\LH$ if and only if there exist constants $M,\kappa>0$ such that
    \begin{equation}\label{ineq_log_tau}
        \frac{1}{p(x)-1}\geq \kappa\log(|x|),
    \end{equation}
    for almost every $x\in\R$, such that $|x|>M$.
\end{prop}
\begin{proof}
    Fix $\p\in\mathcal{P}(\R)$ and let $\tau:\R\to [0,\infty]$ be given by
    \begin{equation*}
        \tau(x)=\begin{cases}
            \frac{1}{p(x)-1}&\text{if }1<p(x)<\infty,\\
            \infty &\text{if } p(x)=1,\\
            0  &\text{if }p(x)=\infty,
        \end{cases}
    \end{equation*}
    for almost every $x\in\R$, so that $p(x)=1+\tau(x)^{-1}$. First assume that $\p$ satisfies inequality \eqref{ineq_log_tau}. This implies that $\tau(x)\geq\kappa\log(|x|)$  for almost every $x\in\R$, such that $|x|>M$. 
    
    Therefore, if $\p$ satisfies inequality \eqref{ineq_log_tau} then for $p_{\infty}=1$ and $M'>M$ sufficiently big, we have that
\begin{align*}
    \left|\frac{1}{p_{\infty}}-\frac{1}{p(x)}\right|&=\frac{1}{q(x)}\\
    &\leq \frac{1}{1+\tau(x)}\\
    &\leq \frac{1}{1+\kappa\log(|x|)}\\
    &\leq \frac{1}{\kappa\log(e+|x|)}=\frac{\kappa^{-1}}{\log(e+|x|)},
\end{align*}
for almost every $|x|>M'>M$, where in the last inequality we used the fact that we may choose $M'>M$ big enough so that $\kappa\log(|x|)+1\geq \kappa\log(e+|x|)$ fo every $|x|>M'$. On the other hand, if $|x|<M'$, since $ \left|1-\frac{1}{p(x)}\right|\leq 1$ and $x\mapsto\log(e+x)$ is increasing for $x\geq0$, we have that
\begin{equation*}
     \left|{1}-\frac{1}{p(x)}\right|\leq \frac{K}{\log(e+|x|)},
\end{equation*}
where $K=\log(e+M')$. Therefore, taking $C=\max\{K,\kappa^{-1}\}$, we have that $ \left|\frac{1}{1}-\frac{1}{p(x)}\right|\leq\frac{C}{\log(e+|x|)}$ for almost every $x\in\R$, proving that $1/\p$ is log-Hölder continuous at infinity with $p_\infty=1$, as claimed. Now suppose that $1/\p$ is  $\LH$. Then 
\begin{align*}
    \frac{1}{p(x)}&\geq 1-\frac{C}{\log(e+|x|)}\\
    &=\frac{\log(e+|x|)-C}{\log(e+|x|)},
\end{align*}
hence 
\begin{align*}
    p(x)&\leq 1+\frac{C}{\log(e+|x|)-C}\\
    &=1+(C^{-1}\log(e+|x|)-1)^{-1}\\
    &\leq 1+(C^{-1}-\varepsilon_0)^{-1}\log(|x|)^{-1},
\end{align*}
for any $\varepsilon_0\in (0,C^{-1})$ and every $|x|>M_{\varepsilon_0}>0$ sufficiently big. Therefore we have that $\p\in\mathcal{P}(\R)$ satisfies inequality \eqref{ineq_log_tau} with $\kappa=(C^{-1}-\varepsilon_0)$ and $M=M_{\varepsilon_0}$.
\end{proof}

Thus in particular, any $\p\in\mathcal{P}(\Omega)$, $p(x)=1+\tau(x)^{-1}$ is $\LH$ if there exist $\kappa,M>0$ such that
\begin{equation}\label{ineq_LH}
    \tau(x)\geq \kappa\log(|x|),
\end{equation}
for almost every $|x|\geq M$.

\begin{obs}
    We note that if $p_{+}<\infty$, then $\frac{1}{\p}$ is log-Hölder continuous at infinity if and only if $\p$ is log-Hölder continuous at infinity, so in this case we may replace the log-Hölder decay condition in Proposition \ref{coro_rapid_decay} by log-Hölder continuity at infinity. Moreover, if one relaxes the conditions for $\p$ to be log-Hölder continuous at infinity by requiring inequality \eqref{ineq_log-holder} to hold only for sufficiently large $|x|$, then one can also prove that in this case $\frac{1}{\p}$ is log-Hölder continuous at infinity and therefore this condition is also sufficient for Proposition \ref{coro_rapid_decay}.
\end{obs}

Next we present a particular case where $1/\p$ is $\LH$ and one can obtain a sharper estimate for $\Psi_f$.

\begin{example}\label{exemp_|x|}
    Suppose that $p(\cdot)\in \mathcal{P}(\R)$ satisfies
    \begin{equation}\label{ineq_p_|x|^k}
       p(x)\leq 1+C_0^{-1}|x|^{-k}=\frac{1+C_0|x|^k}{C_0|x|^k},
    \end{equation}
    for almost every $x\in\R$, and for some $C_0, k> 0$. Then for any $f\in L^{p(\cdot)}(\R)$ we have that
    \begin{equation}\label{ineq_Psi_|x|^k}
        |\Psi_f(s)|\leq K\|f\|_{p(\cdot)}|s|,
    \end{equation}
    for every $s\in\R$, where \begin{align*}
        K&=\delta\left(\R^{\p}_1\right)+\left(2-\frac{1}{p_{+}}\right)\exp\left({k\cdot W\left(\left(\frac{2\Gamma(\frac{1}{k})}{k^{k+1}C_0^{\frac{1}{k}}}\right)^{\frac{1}{k}}\right)}\right)\\
        &\leq K_{\p}\exp\left({k\cdot W\left(\left(\frac{2\Gamma(\frac{1}{k})}{k^{k+1}C_0^{\frac{1}{k}}}\right)^{\frac{1}{k}}\right)}\right).
    \end{align*}
    The proof follows the same arguments from the proof of Proposition \ref{coro_rapid_decay} up to inequality \eqref{ineq_general_tau_q}, except that now inequality \eqref{ineq_p_|x|^k} implies that 
\begin{equation}\label{ineq_q_|x|^k}
      1+C_0|x|^k\leq q(x),
    \end{equation}
   for almost every $x\in\R$. Hence, considering $\lambda>s$, we have that
    \begin{align*}
        \rho_{q(\cdot),\R\backslash\R_{\infty}^{q(\cdot)}}(u_s/\lambda)&\leq \int_{-\infty}^{\infty}\left(\frac{s}{\lambda}\right)^{1+C_0|t|^k}dt\\
        &=\int_{-\infty}^{\infty}e^{-(1+C_0|t|^k)\log(\lambda/s)}dt\\
         &=2\int_{0}^{\infty}e^{-(1+C_0t^k)\log(\lambda/s)}dt\\
         &{=2e^{-\log(\lambda/s)}\int_0^{\infty}e^{-C_0t^k\log(\lambda/s)}dt}\\
         &=\frac{2}{k}e^{-\log(\lambda/s)}C_0^{\frac{1}{k}}\log(\lambda/s)^{\frac{1}{k}}\int_0^\infty e^{-y}y^{\frac{1}{k}-1}dy\\
         &=\frac{s}{\lambda}\frac{2\Gamma(\frac{1}{k})}{k(C_0\log(\lambda/s))^{1/k}}\\
         &=\frac{2\Gamma(\frac{1}{k})}{kC_0^{\frac{1}{k}}}\frac{(\lambda/s)^{-1}}{\log(\lambda/s)^{1/k}},
    \end{align*}
 where in the third-to-last line we performed the change of variables $C_0t^k\log(\lambda/s)\mapsto y$.  Therefore, from the fact that $W(x)>0$ for $x>0$ and from the inequality above we conclude that for 
 \begin{equation*}
     \lambda\geq \exp\left({k\cdot W\left(\left(\frac{2\Gamma(\frac{1}{k})}{k^{k+1}C_0^{\frac{1}{k}}}\right)^{\frac{1}{k}}\right)}\right)s,
 \end{equation*}
   we have that $\rho_{q(\cdot),\R\backslash\R_{\infty}^{q(\cdot)}}(u_s/\lambda)\leq 1$. Hence, setting $K'\defeq \exp\left({k\cdot W\left(\left(\frac{2\Gamma(\frac{1}{k})}{k^{k+1}C_0^{\frac{1}{k}}}\right)^{\frac{1}{k}}\right)}\right)$, we have that $\|u_s\|_{L^{q(\cdot)}(\R\backslash\R_{\infty}^{q(\cdot)})}\leq  K'|s|$. Finally, since $p_{-}=1$ we have that
   \begin{align*}
       |\Psi_f(s)|&\leq \delta\left(\R^{\p}_1\right)\|f\|_{\p}|s|+K_{p(\cdot)}\|f\|_{p(\cdot)}K'|s|\\
       &= \|f\|_{\p}\left(\delta\left(\R^{\p}_1\right)+\left(2-\frac{1}{p_{+}}\right)\exp\left({k\cdot W\left(\left(\frac{2\Gamma(\frac{1}{k})}{k^{k+1}C_0^{\frac{1}{k}}}\right)^{\frac{1}{k}}\right)}\right)\right)|s|,
   \end{align*}
   as claimed.
\end{example}

\begin{obs}\label{remark_asympt}
    Note that if $\p\in\mathcal{P}(\R)$ satisfies $p_{+}<\infty$ or  $1/\p$ is $\LH$, then for any $f\in \Lp$ we have that
    \begin{equation*}
        \Psi_f(s) = O\left(|s|^{\frac{1}{p_{-}}}\right)\text{ as } |s|\to \infty.
    \end{equation*}
    Moreover,
    \begin{equation*}
        \Psi_f(s)=\left\{\begin{alignedat}{3}
      &O(|s|) &&\text{ as } |s|\to 0^+, \text{ if $1/\p$ is $\LH$},\\
      &O\left(|s|^{\frac{1}{p_{+}}}\right) &&\text{ as } |s|\to 0^+, \text{ if }p_{+}<\infty.
        \end{alignedat}\right.
    \end{equation*}
\end{obs}
\begin{corollary}
    Let $\p\in \mathcal{P}(\R)$. Then for any $f\in L^{p(\cdot)}(\R)$, we have that $\Psi_f$ is a Lipschitz continuous function if $\frac{1}{\p
    }$ is $\LH$, or if $p_+<\infty$ then $\Psi_f$ is a Hölder continuous function with exponent $\frac{1}{p_{+}}$.
\end{corollary}
\begin{proof}
    By Remark \ref{remark_asympt}, it is clear that $\Psi_f$ possess the claimed type of continuity at $0$. To verify continuity elsewhere, notice that  for every non-zero $s,h\in\R$ we have that
    \begin{align*}
        |\Psi_f(s+h)-\Psi_f(s)|&=\left|\int_{-\infty}^{\infty}\left[\left(\frac{1-e^{-i(s+h)t}}{it}\right)-\left(\frac{1-e^{-ist}}{it}\right)\right]f(t)dt\right|\\
        &=\left|\int_{-\infty}^{\infty}e^{-ist}\left(\frac{1-e^{-iht}}{it}\right)f(t)dt\right|\\
        &=|\Psi_{g_s}(h)|,
    \end{align*}
    where $g_s(x)\defeq e^{-isx}f(x)$ is also in $\Lp$ and satisfies $\|g_s\|_{\p}=\|f\|_{\p}$. Therefore the continuity from the statement follows once again from Remark \ref{remark_asympt}.
\end{proof}


\section{The Fourier transform and the spaces \texorpdfstring{$\mathcal{B}_{p(\cdot)}(\R)$}{} and \texorpdfstring{$\mathcal{A}_{\p}(\R)$}{}}\label{sec-fourier-transf}

In this section we provide an alternative definition for the Fourier transform of functions in $\Lp$,  inspired by the one given in \cite{Talvila2025} for functions in $L^p(\R)$. We also show that the set of all such Fourier transforms, equipped with an appropriate norm, is a Banach space isomorphic to $\Lp$.

\begin{theorem}\label{theo_1}
     Let $\p\in \mathcal{P}(\R)$ satisfy $p_{+}<\infty$ or $1/\p$ is $\LH$.
     \begin{enumerate}
         \item Define $\mathcal{B}_{p(\cdot)}(\R) = \{\Psi_f : f\in L^{p(\cdot)}(\R)\}$. For $f\in L^{\p}(\R)$, let $\|\Psi_f\|_{\mathcal{B}_{\p}(\R)}\defeq\|f\|_{\p}$. Then  $\left(\mathcal{B}_{p(\cdot)}(\R),\|\cdot\|_{\mathcal{B}_{\p}(\R)}\right)$ is a Banach space isometrically isomorphic to $L^{\p}(\R)$.
         \item Define $\mathcal{A}_{\p}(\R) = \{\Psi_f': f\in L^{\p}(\R)\}$, where the superscript $'$ denotes the distributional derivative. For $f\in L^{\p}(\R)$, let $\|\Psi_f'\|_{\mathcal{A}_{\p}(\R)}\defeq\|f\|_{\p}$. Then 
         \noindent $\left(\mathcal{A}_{\p}(\R),\|\cdot\|_{\mathcal{A}_{\p}(\R)}\right)$ is a Banach space isometrically isomorphic to $\Lp$.
     \end{enumerate}
\end{theorem}
\begin{proof}
    {\it (1)} Consider the mapping $\Lambda: L^{p(\cdot)}(\mathbb{R}) \rightarrow \mathcal{B}_{\p}(\mathbb{R})$ given by $\Lambda(f)=\Psi_f$. Note that $\Lambda$ is linear and onto $\mathcal{B}_{\p}(\mathbb{R})$. The statement will follow once we prove that $\Lambda$ is injective, which guarantees that $\|\cdot\|_{\mathcal{B}_{\p}(\R)}$ is a norm. To show it is injective suppose $f \in L^{p(\cdot)}(\mathbb{R})$ and $\Psi_f(s)=0$ for all $s \in \mathbb{R}$. For each $\varphi \in \mathcal{S}(\mathbb{R})$ applying Fubini's Theorem (which will be justified below) and integrating by parts yields
\begin{align}\label{eq_Fubini}
0 & =\int_{-\infty}^{\infty} \Psi_f(s) \varphi^{\prime}(s) {d} s=\int_{-\infty}^{\infty} f(t) \int_{-\infty}^{\infty}\left(\frac{1-{e}^{-{i} s t}}{{i} t}\right) \varphi^{\prime}(s) {d} s {~d} t \\
& =-\int_{-\infty}^{\infty} f(t) \int_{-\infty}^{\infty} {e}^{-{i} s t} \varphi(s) {d} s {~d} t=-\int_{-\infty}^{\infty} f(t) \widehat{\varphi}(t) {d} t.  \notag
\end{align}
Notice that Gaussian functions are in the Schwartz space and every Gaussian is the Fourier transform of a Gaussian. Therefore, $\int_{-\infty}^{\infty} f(t) \Theta_a(x-t) {d} t=0$ for each $a>0$, where $\Theta_a(t)={e}^{-t^2 /(4 a)} / \sqrt{4 \pi a}$ is the heat kernel. Then $f * \Theta_a(x)=0$ for each $x \in \mathbb{R}$ and $a>0$. On the other hand, by Theorem 5.8 of \cite{Cruz-Uribe} we have that $f * \Theta_a(x)\to f$ pointwise almost everywhere as $a\to 0^+$. Therefore it follows that $f\equiv0$. We conclude that $\|\cdot\|_{\mathcal{B}_{\p}(\R)}$ is indeed a norm and therefore $\Lambda$ is a surjective isometry between $\mathcal{B}_{p(\cdot)}(\R)$ and $L^{p(\cdot)}(\R)$, so these spaces are isomorphic.

We now justify the application of Fubini's Theorem in \eqref{eq_Fubini}. Using the estimates $\left|u_s(t)\right| \leq|s|$ for $|t| \leq 1$, and $\left|u_s(t)\right| \leq 2 /|t|$ for $|t| \geq 1$, we have that
\begin{align}
\int_{\R}\int_{\R}|u_s(t)||f(t)||\varphi'(s)|dt\ ds&\leq \int_{\R}\int_{|t|\leq 1}|s||f(t)||\varphi'(s)|dt\ ds+\int_{\R}\int_{|t|>1}\frac{2}{|t|}|f(t)||\varphi'(s)|dt\ ds\notag\\
&= \int_{\R}|s||\varphi'(s)|ds\int_{\R}\chi_{\{|t|\leq 1\}}(t)|f(t)|dt\notag\\
&+2\int_{\R}|\varphi'(s)|ds\int_{\R}\frac{\chi_{\{|t|>1\}}}{|t|}|f(t)|dt\notag\\
&\leq C_1K_{p(\cdot)}\|\chi_{\{|t|\leq 1\}}\|_{q(\cdot)}\|f\|_{p(\cdot)}+C_2K_{p(\cdot)}\left\|\frac{\chi_{\{|t|>1\}}}{|t|}\right\|_{q(\cdot)}\|f\|_{p(\cdot)},\label{ineq_Fubini_1}
\end{align}
for $0<C_1=\|s\varphi'(s)\|_{L^1(\R)},C_2=2\|\varphi'\|_{L^1(\R)}<\infty$, where in the last line we applied the generalized Hölder's inequality.
Notice that by \cite[Lemma 2.39]{Cruz-Uribe} we have that $\chi_{[-1,1]}$ is in $L^{\q}(\R)$. Similarly, if $1<p_{+}<\infty$, then 
\begin{align*}
    \rho_{q(\cdot)}\left(\frac{\chi_{\{|t|> 1\}}}{|t|}\right)&\leq\int_{\{|t|>1\}\backslash\R^{\q}_{\infty}}|t|^{-q(t)}dt+\delta\left(\R_\infty^{\q}\right)\\
    &\leq \int_{\{|t|>1\}}|t|^{-q_{-}}dt+1\\
    &=2(q_{-}-1)^{-1}+1<\infty.
\end{align*}
And so $\frac{\chi_{\{|t|> 1\}}}{|t|}$ is in $L^{\q}(\R)$ as well. If $p_+=1$, the $\q\equiv\infty$ and so clearly $\chi_{[-1,1]}$ is in $L^{\q}(\R)$.
In the case where $1/\p$ is $\LH$, by inequality \eqref{ineq_LH} we have that
\begin{align*}
     \rho_{q(\cdot)}\left(\frac{\chi_{\{|t|> 1\}}}{|t|}\right)&\leq\int_{\{|t|>1\}\backslash\R^{\q}_{\infty}}|t|^{-q(t)}dt+\delta\left(\R_\infty^{\q}\right)\\
    &\leq 2\int_{1}^Mt^{-1}dt + 2\int_M^\infty t^{-(1+\kappa\log(|t|))}dt+1
    <\infty,
\end{align*}
for some $\kappa,M>0$. Therefore $\frac{\chi_{\{|t|> 1\}}}{|t|}$ is in $L^{\q}(\R)$ as well. Applying these facts to inequality \eqref{ineq_Fubini_1}, we conclude that the integral is finite, justifying the application of Fubini's Theorem.

{\it (2)} Notice that the distributional derivative provides a linear surjective mapping from $\mathcal{B}_{p(\cdot)}(\R)$ to $\mathcal{A}_{p(\cdot)}(\R)$. Moreover, it is injective due to the fact that the estimates $|\Psi_f(s)|\lesssim \|f\|_{\p}|s|$ or $|\Psi_f(s)|\lesssim \|f\|_{\p}|s|^{\frac{1}{p_{+}}}$ as $|s|\to0^+$ (see Remark \ref{remark_asympt}) imply that the only constant in $\mathcal{B}_{\p}(\R)$ is $\Psi_f\equiv0$. We conclude from the definition that $\|\cdot\|_{\mathcal{A}_{\p}(\R)}$ is indeed a norm and therefore $\mathcal{B}_{p(\cdot)}(\R)$, $\mathcal{A}_{p(\cdot)}(\R)$ and $\Lp$ are all isomorphic to each other.
\end{proof}

\begin{definition}\label{definition_Fourier}
    Consider $\p\in\mathcal{P}(\R)$ such that $p_{+}<\infty$ or $1/\p$ is $\LH$. Define the Fourier transform as $\mathcal{F}:\Lp\to \mathcal{A}_{\p}(\R)$ by $\mathcal{F}(f)\equiv\widehat{f}\defeq\Psi_f'$.
    Or equivalently,
    $$\langle \widehat{f},\varphi\rangle=-\langle \Psi_f,\varphi'\rangle=-\int_{-\infty}^{\infty}\Psi_f(s)\varphi'(s)ds,$$ for each $\varphi\in\mathcal{S}(\R)$. The space of Fourier transforms on $L^{\p}(\R)$ is then denoted by $\mathcal{A}_{\p}(\R)$. 
\end{definition}

Notice that the definition above can be seen as an extension of the definition given in \cite{Talvila2025} for constant $p\neq \infty$.

\section{Properties of the Fourier transform}\label{sec-prop-fourier}

First note that as in \cite{Talvila2025}, since $\Psi_f$ is continuous and of polynomial growth if $f\in \Lp$ such that $p_+<\infty$ or $1/\p$ is $\LH$, then $\widehat{f}$ defined in Definition \ref{definition_Fourier} is a tempered distribution. Therefore the elements of $\mathcal{A}_{\p}(\R)$ inherit the properties of the Fourier transforms in $\mathcal{S}'(\R)$. In this section we will also prove other properties related to our definition of the Fourier transform.

\begin{prop}
    Definition \ref{definition_Fourier} agrees with the definition of the Fourier transform of a tempered distribution.
\end{prop}
\begin{proof}
    Indeed, as in the proof of Theorem \ref{theo_1}, we have that
\begin{align*}
\langle\widehat{f},\varphi\rangle&= -\int_{-\infty}^{\infty} \Psi_f(s) \varphi^{\prime}(s) {d} s=-\int_{-\infty}^{\infty} f(t) \int_{-\infty}^{\infty}\left(\frac{1-{e}^{-{i} s t}}{{i} t}\right) \varphi^{\prime}(s) {d} s {~d} t \\
& =\int_{-\infty}^{\infty} f(t) \int_{-\infty}^{\infty} {e}^{-{i} s t} \varphi(s) {d} s {~d} t=\int_{-\infty}^{\infty} f(t) \widehat{\varphi}(t) {d} t=\langle f,\widehat{\varphi}\rangle.
\end{align*}
\end{proof}
\begin{prop}
     Consider $\p\in\mathcal{P}(\R)$ such that $p_{+}<\infty$ or $1/\p$ is $\LH$. Let $n\in\N$ and $F^{(n-1)}$ be absolutely continuous such that $F^{(n)} \in L^{\p}(\mathbb{R})$. For each $0 \leqslant k \leqslant$ $n-2$, suppose that $\lim _{|t| \rightarrow \infty} F^{(k)}(t) / t=0$. Then $\widehat{F^{(n)}}=\Psi_{F^{(n)}}^{\prime}$, where
\begin{equation}\label{eq_F^(n)}
\Psi_{F^{(n)}}(s)=\int_{-\infty}^{\infty} \frac{n!}{{i} t^{n+1}}\left[1-{e}^{-{i} s t} \sum_{k=0}^n \frac{({i} s t)^k}{k!}\right] F(t) {d} t.
\end{equation}
\end{prop}
\begin{proof}
      First notice that induction on $n\in\N_0$ establishes the formula
    $$
    u_s^{(n)}(t)=\frac{(-1)^n n!}{{i} t^{n+1}}\left[1-{e}^{-{i} s t} \sum_{k=0}^n \frac{({i} s t)^k}{k!}\right].
    $$
    Notice that $u_s$ and its derivatives are continuous and of decay like $1 / t$ as $|t| \rightarrow \infty$. Equality \eqref{eq_F^(n)} follows after an application of integration by parts $n$ times on $\Psi_{F^{(n)}}(s)=\int_{-\infty}^{\infty} u_s(t) F^{(n)}(t) {d} t$,  since $\lim_{|x|\to\infty}F^{(n-1)}(x)/x=0$, as shown next. Indeed, notice that because $F^{(n-1)}$ is absolutely continuous, we have that $F^{(n-1)}(x)=F^{(n-1)}(0)+\int_0^x F^{(n)}(t) {d} t$. Applying the generalized Hölder's inequality then gives 
    \begin{equation*}
        \left|F^{(n-1)}(x)\right| \leq\left|F^{(n-1)}(0)\right|+K_{\p}\left\|F^{(n)}\right\|_{\p}\|\chi_{[0,x]}\|_{q(\cdot)},
    \end{equation*}
    for every $x\in\R$. Let us estimate $\|\chi_{[0,x]}\|_{q(\cdot)}$ for large $x$. Suppose without loss of generality that $x> 0$. Note that for $\lambda>0$ we have that
    \begin{align*}
\rho_{\q}\left(\chi_{[0,x]}/\lambda\right)&=\int_{\R\backslash \R^{\q}_{\infty}}\left(\frac{\chi_{[0,x]}(t)}{\lambda}\right)^{q(t)}dt+\frac{1}{\lambda}\|\chi_{[0,x]}\|_{L^\infty(\R^{\q}_\infty)}\\
&\leq \int_0^x\lambda^{-q(t)}dt+\frac{1}{\lambda}.
    \end{align*}
    Assume first that $p_{+}<\infty$. If $p_{+}=1$, then $\q\equiv \infty$, so $\|\chi_{[0,x]}\|_{q(\cdot)}=1$. Next assume that $p_{+}>1$. Then considering $1\leq \lambda\leq |x|^{p_{+}-1}$, we have that 
    \begin{align*}
        \rho_{\q}\left(\chi_{[0,x]}/\lambda\right)&\leq \int_0^x\lambda^{-q_{-}}dt+\lambda^{-1}\\
        &=\lambda^{-q_{-}}\left(|x|+\lambda^{(q_{-}-1)}\right)\leq \lambda^{-q_{-}}2|x|,
    \end{align*}
    where we used the fact that  $q_{-}-1=\frac{1}{p_{+}-1}$ and so $\lambda^{(q_{-}-1)}\leq \left(|x|^{p_{+}-1}\right)^{\frac{1}{p_{+}-1}}=|x|$.
    Therefore, if $|x|^{p_{+}-1}\geq \lambda\geq \max\left\{2^{\frac{1}{q_{-}}}|x|^{\frac{1}{q_{-}}},1\right\}$ we have that $\rho_{\q}(\chi_{[0,x]}/\lambda)\leq 1$. 
    If $|x|$ is large enough, then clearly $2^{\frac{1}{q_{-}}}|x|^{\frac{1}{q_{-}}}\geq 1$ and $|x|^{p_{+}-1}\geq 2^{\frac{1}{q_{-}}}|x|^{\frac{1}{q_{-}}}=2^{\frac{1}{q_{-}}}|x|^{\frac{p_{+}-1}{p_{+}}}$.
    Therefore $\|\chi_{[0,x]}\|_{\q}\leq 2^{\frac{1}{q_{-}}}|x|^{\frac{1}{q_{-}}}$ for $|x|$ large enough. This implies that
    \begin{align*}   \lim_{|x|\to\infty}\frac{|F^{(n-1)}(x)|}{|x|}&\leq \lim_{|x|\to\infty}\frac{\left|F^{(n-1)}(0)\right|+K_{\p}\left\|F^{(n)}\right\|_{\p}\|\chi_{[0,x]}\|_{q(\cdot)}}{|x|}\\
    &\leq \lim_{|x|\to\infty}K_{\p}\|F^{(n)}\|_{\p}2^{q_{-}}|x|^{\frac{1}{q_{-}}-1}=0.
    \end{align*}
  It follows that $\lim _{|t| \rightarrow \infty} F^{(n-1)}(t) / t=0$.
    Finally, suppose that $1/\p$ is $\LH$ . Then there exist $M,\kappa>0$,  such that $p(x)\leq 1+\kappa^{-1}\log(|x|)^{-1}$ for almost every $|x|>M$. Therefore for $\lambda>e^{2\kappa^{-1}}$, we have that
    \begin{align*}
\rho_{\q}\left(\chi_{[0,x]}/\lambda\right)&\leq \int_{[0,x]\backslash\R^{\q}_1}\lambda^{-q(t)}dt+\lambda^{-1}\delta\left(\R^{\q}_\infty\cap [0,x]\right)\\
        &\leq \int_0^{M}\lambda^{-1}dt+  \lambda^{-1}\int_M^\infty e^{-\kappa\log(t)\log(\lambda)}dt+\lambda^{-1}\\
        &\leq \lambda^{-1}\left(M+\int_M^\infty t^{-\kappa\log(\lambda)}dt+1\right)\\
        &\leq \lambda^{-1}\left(M+\int_M^\infty t^{-2} dt+1\right)\\
        &=C'\lambda^{-1},
  \end{align*}
  for $C'=M+(2M)^{-1}+1<\infty$.  
  If $\lambda\geq C''=\max\{e^{2\kappa^{-1}},C'\}$ we have that  $\rho_{\q}(\chi_{[0,x]}/\lambda)\leq 1$, and therefore $\|\chi_{[0,x]}\|_{\q}\leq C''$. This implies that
    \begin{align*}   \lim_{|x|\to\infty}\frac{|F^{(n-1)}(x)|}{|x|}&\leq \lim_{|x|\to\infty}\frac{\left|F^{(n-1)}(0)\right|+K_{\p}\left\|F^{(n)}\right\|_{\p}\|\chi_{[0,x]}\|_{q(\cdot)}}{|x|}\\
    &\leq \lim_{|x|\to\infty}K_{\p}\|F^{(n)}\|_{\p}C''|x|^{-1}=0,
    \end{align*}
    finishing the proof.
\end{proof}
\begin{corollary}
     Consider $\p\in\mathcal{P}(\R)$ such that $p_{+}<\infty$ or $1/\p$ is $\LH$. Let $f$ be absolutely continuous such that $f' \in L^{\p}(\mathbb{R})$. Then $\widehat{f'}=\Psi_{f'}^{\prime}$, where
\begin{equation}
\Psi_{f'}(s)=\int_{-\infty}^{\infty} \frac{1}{{i}t^2}\left[1-{e}^{-{i} s t}(1+ist)\right] f(t) {d} t.
\end{equation}
\end{corollary}
\begin{prop}\label{prop_properties}
    Consider $\p\in\mathcal{P}(\R)$ such that $p_{+}<\infty$ or $1/\p$ is $\LH$ and fix  $f\in\Lp$. For any function $g:\R\to\mathbb{C}$, denote $\tau_ag(x)=g(x-a)$ and $\tilde{g}(x)=g(-x)$, for every $a,x\in\R$. Then
    \begin{enumerate}
               \item  The Fourier transform is a bounded linear isomorphism and isometry $\mathcal{F}$ : $L^{\p}(\mathbb{R}) \to \mathcal{A}_{\p}(\mathbb{R})$, that is, $\|\hat{f}\|_{\mathcal{A}_{\p}(\R)}=\|f\|_{\p}$;

\item $\hat{f}^{(n)}=\Psi_f^{(n+1)}$;
\item For $a \in \mathbb{R}$, we have that  $\widehat{\tau_a f}=\Psi_{\tau_a f}'$, and $\Psi_{\tau_a f}(s)=f * \tilde{u}_s(-a)$;
\item For $a \in \mathbb{R}$ and $F(t)={e}^{{i} a t} f(t)$, we have that $\widehat{F}=\left(\tau_a \Psi_f\right)^{\prime}$;
\item $\hat{\tilde{f}}=\Psi_{\tilde{f}}^{\prime}$, and $\Psi_{\tilde{f}}=-\widetilde{\Psi}_f$;
\item If $\hat{f} \in L^{r(\cdot)}(\mathbb{R})$ for some $r(\cdot)\in \mathcal{P}(\R)$ such that $r_{+}<\infty$ or $1/r(\cdot)$ is $\LH$, then $\hat{\hat{f}}=\Psi_{\widehat{f}}'$;
\item Let $g(x)=f(a x+b)$ for $a, b \in \mathbb{R}$ with $a \neq 0$. Then $\hat{g}=\Psi_g^{\prime}$, where $\Psi_g(s)=$ $\operatorname{sgn}(a) \Psi_{\tau_{-b} f}(s / a)$.
    \end{enumerate}
\end{prop}
\begin{proof}
    Claim {\it (1)} follows from Definition \ref{definition_Fourier} and Theorem \ref{theo_1}.
Claims \noindent{\it (2), (3)} and {\it (6)} 
follow immediately from the definitions.

\noindent For the proof of {\it (4)}, notice that
$$
\Psi_F(s)=\int_{-\infty}^{\infty}\left(\frac{1-{e}^{-{i}(s-a) t}}{{i} t}-\frac{1-{e}^{{i} a t}}{{i} t}\right) f(t) {d} t=\Psi_f(s-a)-\Psi_f(-a),
$$
therefore $\widehat{F}=\left(\tau_a \Psi_f\right)^{\prime}$.

\noindent For {\it (5)}, note that from the definition of $\tilde{f}$ and performing a change of variables, we have that
$$
\Psi_{\tilde{f}}(s)=\int_{-\infty}^{\infty}\left(\frac{1-{e}^{-{i} s t}}{{i} t}\right) f(-t) {d} t=-\int_{-\infty}^{\infty}\left(\frac{1-{e}^{{i} s t}}{{i} t}\right) f(t) {d} t.
$$
\noindent Finally, {\it (7)} follows from a change of variables in the integral which defines $\Psi_g$.
\end{proof}

\begin{example} 
Here we present a couple of examples of  conditions which guarantee that $\widehat{f}\in\Lp$, as in Proposition \ref{prop_properties} {(6)}. The ideas for these examples were adapted from \cite{Talvila2025}.

    First suppose that $f$ is absolutely continuous and $f,f'\in L^1(\R)$. This conditions imply that $\widehat{f},\widehat{f'}$ are continuous and bounded, and $\widehat{f}(s)=\widehat{f'}(s)/is$. Hence $\widehat{f}\in L^{\p}(\R)$ for every $\p\in\mathcal{P}(\R)$ satisfying $1<p_{-},p_{+}<\infty$. Indeed, notice that in this case
    \begin{align*}
        \rho_{\p}\Big(\hat{f}\Big)&=\int_{-1}^1 |\widehat{f}(x)|^{p(x)}dx+\int_{\{|x|>1\}}|\widehat{f}(x)|^{p(x)}dx\\
        &\leq 2\max\{\|\widehat{f}\|_{\infty}^{p_{-}},\|\widehat{f}\|_{\infty}^{p_{+}}\}
        +\int_{\{|x|>1\}}\left|\frac{\widehat{f'}(x)}{x}\right|^{p(x)}dx\\
        &\leq C_f+\int_{\{x:|x|>1 \text{ and }|\widehat{f'}(x)/x|\geq 1\}} \left|\frac{\widehat{f'}(x)}{x}\right|^{p_{+}}dx+\int_{\{x:|x|>1 \text{ and }|\widehat{f'}(x)/x|< 1\}} \left|\frac{\widehat{f'}(x)}{x}\right|^{p_{-}}dx\\
        &\leq C_f+\int_{\{|x|>1\}} \left|\frac{\widehat{f'}(x)}{x}\right|^{p_{+}}dx+\int_{\{|x|>1\}} \left|\frac{\widehat{f'}(x)}{x}\right|^{p_{-}}dx\\
        &\leq C_f+\int_{\{|x|>1\}} \frac{\|\widehat{f'}\|^{p_+}_\infty}{|x|^{p_{+}}}dx+\int_{\{|x|>1\}} \frac{\|\widehat{f'}\|^{p_-}_\infty}{|x|^{p_{-}}}dx<\infty,
    \end{align*}
    where $C_f= 2\max\{\|\widehat{f}\|_{\infty}^{p_{-}},\|\widehat{f}\|_{\infty}^{p_{+}}\}$, and where in the last line we used the fact that $\widehat{f'}$ is bounded and $1<p_{-},p_{+}$. This proves that $\widehat{f}\in \Lp$.

    Next, suppose that $f\in L^1(\R)$ and is of bounded variation. Integration by parts implies that $\widehat{f}(s)=(1/(is))\int_{-\infty}^\infty e^{-ist}df(t)$. Then $|\widehat{f}(s)|\leq Vf/|s|$, where $Vf<\infty$ denotes the total variation of $f$. We claim that $\widehat{f}\in \Lp$ for $1<p_{-},p_{+}<\infty$. Indeed, note that since $\widehat{f}$ is continuous, we have that
    \begin{align*}
        \rho_{\p}(\widehat{f})&=\int_{-\infty}^\infty|\widehat{f}(x)|^{p(x)}dx\\
        &\leq \int_{\{|x|>1\}}\frac{(Vf)^{p(x)}}{|x|^{p(x)}}dx+\int_{\{|x|\leq1\}}\|\widehat{f}\|_{\infty}^{p(x)}dx.
    \end{align*}
  Therefore
\begin{equation*}
    \rho_{\p}(\widehat{f})\leq \max\{(Vf)^{p_{-}},(Vf)^{p_{+}}\}\int_{|x|>1}\frac{1}{|x|^{p_{-}}}dx+2\max\{\|\widehat{f}\|_{\infty}^{p_{-}},\|\widehat{f}\|_{\infty}^{p_{+}}\}<\infty,
\end{equation*}
and so $\widehat{f}\in\Lp$.
\end{example}

\section{Integration in \texorpdfstring{$\mathcal{A}_{\p}(\R)$}{}}\label{sec-integration-fourier}

In this section we define a way to integrate the Fourier transform of $f\in\Lp$, and prove certain formulas involving such integrals which extend the case where $\widehat{f}$ is integrable, such as an exchange and inversion formula.

Due to the fact that $\Psi_f$ is continuous, $\widehat{f}$ is integrable in the sense of the continuous primitive integral, see \cite{Talvila2007}. According to this theory, for $f=F'$, where $F$ is a continuous functions on an interval $[a,b]$, then $\int_a^b F'g=F(b)g(b)-F(a)g(a)-\int_a^bF(x)dg(x)$, where $g\in \mathcal{B}V([a,b])$ and the last integral is a Henstock-Stieltjes integral. Here $\mathcal{B}V(I)$ denotes the set of functions of bounded variation on a closed interval $I\subset \R$. Also, for $g\in\mathcal{B}V(\R)$, let $g(\infty)\defeq\lim_{x\to\infty}g(x)$, and similarly for $g(-\infty)$. Following the notation in \cite{Talvila2007}, we  indicate Lebesgue/Henstock-Stieltjes integrals by explicitly stating the measure and variable of integration, while continuous primitive integrals will show neither.

\begin{definition}\label{definitionintegral}
     Consider $\p\in\mathcal{P}(\R)$ such that $p_{+}<\infty$ or $1/\p$ is $\LH$, and $f\in L^{\p}(\R)$.
    \begin{enumerate}
        \item For $-\infty<a<b<\infty$ and $g\in\mathcal{B}V(\R)$, define
        \begin{equation*}
\int_a^b\widehat{f}g=\int_a^b\Psi_f'g\defeq\Psi_f(b)g(b)-\Psi_f(a)g(a)-\int_a^b\Psi_f(t)dg(t).
        \end{equation*}
        \item Let $g\in\BV(\R)$ such that $g(x)=o(|x|^{-1/p_{-}})$ as $|x|\to \infty$. Then 
        \begin{equation*}
\int_{\infty}^{\infty}\widehat{f}g=\int_{-\infty}^{\infty}\Psi_f'g\defeq-\int_{-\infty}^\infty \Psi_f(t)dg(t).
        \end{equation*}
        \item Let $g\in\BV([\delta,a])$ for each $0<\delta<a<\infty$ such that $g(x)=o(x^{-1/p_{+}})$ as $x\to 0^+$ if $\p$ satisfies $p_+<\infty$, or $g(x)=o(x^{-1})$ as $x\to 0^+$ if $1/\p$ is $\LH$. Then 
        \begin{equation*}
            \int_0^a \widehat{f}g=\int_0^a \Psi'_fg\defeq\Psi_f(a)g(a)-\int_0^a\Psi_f(t)dg(t).
        \end{equation*}
    \end{enumerate}
\end{definition}

By the definition above, notice that in particular 
        \begin{equation*}
            \int_a^b\widehat{f}=\Psi_f(b)-\Psi_f(a).
        \end{equation*}
\begin{obs}
    This definition is consistent with the similar formulas which hold whenever $\widehat{f}$ is a function of appropriate regularity, and the integrals are defined in the classical sense.
\end{obs}

\begin{theorem}[Exchange formula]\label{theo_exchange}
     Consider $\p\in\mathcal{P}(\R)$ such that $p_{+}<\infty$ or $1/\p$ is $\LH$, and $\q$ its conjugate exponent. Let $f\in \Lp$ and $g\in \BV(\R)$ with $g(\pm\infty)=0$, such that $\int_{-\infty}^{\infty}|s|^{1/p_{-}}|dg(s)|<\infty$. Then $\widehat{g}\in L^{\q}(\R)$ and 
    \begin{equation*}
        \int_{-\infty}^\infty\widehat{f}g=\int_{-\infty}^\infty f(s)\widehat{g}(s)ds.
    \end{equation*}
    Moreover, 
    \begin{equation}\label{inequality_statement_exchange}
        \left|\int_{-\infty}^\infty\widehat{f}g\right|\leq\|f\|_{\p}\left(K_1\int_0^M|s|^{1/r}|dg(s)|+K_2\int_{M}^{\infty}|s|^{1/p_{-}}|dg(s)|\right),
    \end{equation}
    for some $K_1,K_2>0$ and $r\geq 1$ depending only on $\p$.
\end{theorem}
\begin{proof}
    Without loss of generality we may assume that $g$ has support contained in $[0,\infty]$. Since $g\in\BV(\R)$, we can write $g=g_1-g_2$, where $g_1,g_2:[0,\infty]\to \R$ decrease to $0$ and $g_1(x)+g_2(x)=Vg(x)$ (see \cite{Cruz-Uribe-Stieltjes}). Therefore 
    \begin{equation}\label{eq_equality_dgs}
\int_{0}^{\infty}|s|^{1/p_{-}}|dg(s)|=-\int_0^{\infty}s^{1/p_{-}}dg_1(s)-\int_0^\infty s^{1/p_{-}}dg_2(s).
    \end{equation}
    Notice that 
    for $x>0$, we have that
    \begin{equation*}
        0\leq x^{1/p_{-}}g_1(x)\leq -\int_x^\infty s^{1/p_{-}}dg_1(s),
    \end{equation*}
    which tends to 0  as $x\to\infty$,
    since the integral is finite for every $x>0$ by \eqref{eq_equality_dgs}.
    Applying a similar argument to $g_2$, and using that $g=g_1+g_2$, we conclude that $g(s)=o(|s|^{1/p_{-}})$. Integrating as in Definition \ref{definitionintegral} we have that
    \begin{align*}
        \int_{-\infty}^\infty\widehat{f}g=\int_{-\infty}^\infty\Psi_f'g&=-\int_{-\infty}^\infty\Psi_f(s)dg(s)\\
        &=-\int_{-\infty}^\infty\int_{-\infty}^\infty\left(\frac{1-e^{-ist}}{it}\right)f(t){d}t\, dg(s)\\
        &=-\int_{-\infty}^\infty f(t)\int_{-\infty}^\infty\left(\frac{1-e^{-ist}}{it}\right)dg(s)dt.
    \end{align*}
    Where in the last line we applied Fubini's theorem,  justified by the fact that
\begin{equation}\label{bound_exchange}
        \int_{-\infty}^\infty|\Psi_f(s)||dg(s)|\leq \|f\|_{\p}\left(K_1\int_0^M|s|^{1/r}|dgs(s)|+K_2\int_{M}^{\infty}|s|^{1/p_{-}}|dg(s)|\right)<\infty,
    \end{equation}
    for some $K,M>0$ and $r\geq 1$ depending on $\p$, where we have used the fact that $|\Psi_f(s)|\leq K\|f\|_{\p}|s|^{\frac{1}{p_{-}}}$ for large $|s|$, and that $|\Psi_f(s)|\leq K\|f\|_{\p}|s|^{\frac{1}{p_{+}}}$ or $|\Psi_f(s)|\leq K\|f\|_{\p}|s|$ for small $|s|$. 
    Then, upon integrating by parts, we have that
    \begin{align*}
        \int_{-\infty}^\infty\widehat{f}g&=-\int_{-\infty}^\infty f(t)\left\{\left[\left(\frac{1-e^{-ist}}{it}\right)g(s)\right]_{s=-\infty}^{s=\infty}-\int_{-\infty}^\infty e^{-ist}g(s)ds\right\}dt\\
        &=\int_{-\infty}^{\infty}f(t)\widehat{g}(t)dt.
    \end{align*}
    Inequality \eqref{inequality_statement_exchange} then follows from \eqref{bound_exchange}. Then notice that inequality \eqref{inequality_statement_exchange} implies that the function $g$ defines a bounded linear functional on $\Lp$ via 
    \begin{equation*}
        f\mapsto\int_{-\infty}^\infty\widehat{f}g\equiv\int_{-\infty}^\infty\Psi_f(s)dg(s),
    \end{equation*}and therefore by Proposition 2.79 in \cite{Cruz-Uribe}, we have that $g\in L^{\q}(\R)$.
\end{proof}
\begin{example}
    Suppose $1<p_{-},p_{+}<\infty$. Let 
    \begin{equation*}
        g(x)=\begin{cases}
            x^{-\alpha}&\text{if}\quad x>1,\\
            0&\text{otherwise,}
        \end{cases}
    \end{equation*}
    where $\alpha p_{-}>1$ and $0<\alpha<1$. Then $g\in \Lp\cap \BV(\R)$. Indeed, it is easy to verify that $g\in \BV(\R)$. Also, note that
    \begin{align*}
        \rho_{\p}(g)&= \int_1^\infty x^{-\alpha p(x)}dx\\
        &\leq \int_1^\infty x^{-\alpha p_{-}}dx\\
        &=\frac{1}{\alpha p_{-}-1}<\infty.
    \end{align*}
    Therefore $g\in\Lp$. Furthermore, notice that
    \begin{equation*}
        \int_{-\infty}^\infty |s|^{1/p_{-}}|dg(s)|=\alpha\int_1^{\infty}s^{-(\alpha+1-1/p_{-})}ds<\infty.
    \end{equation*}
   It follows from Theorem \ref{theo_exchange} that $\widehat{g}\in L^{q(\cdot)}(\R)$. Indeed we can verify this directly: for $s>0$ we have that
    \begin{equation*}
        \widehat{g}(s)=s^{\alpha-1}\int_s^\infty e^{-it}t^{-\alpha}dt\sim s^{\alpha-1}\int_0^\infty e^{-it}t^{-\alpha}dt,
    \end{equation*}
     as $s\to 0^+$. Integration by parts twice shows that $\widehat{g}(s)=e^{-is}/(is)+O(s^{-2})$ as $s\to \infty$.
    Therefore, setting $K=\left|\int_0^\infty e^{-it}t^{-\alpha}dt\right|<\infty$, we have that
    \begin{align*}
        \rho_{\q}(\widehat{g})&=\int_{-\infty}^\infty|\widehat{g}(s)|^{q(s)}ds\\
        &\sim \int_{|s|\leq1}|s|^{(\alpha-1)q(s)}K^{q(s)}ds+\int_{|s|>1}|s|^{-q(s)}ds+K'\\
        &=\max\{K^{q_{-}},K^{q_{+}}\}2\int_0^1 s^{(\alpha-1)q_{+}}ds+2\int_1^\infty s^{-q_{-}}ds+K',
    \end{align*}
    for some $K'<\infty$. But note that since $\alpha p_{-}>1$ and $1/p_{-}+1/q_{+}=1$, we have that $0>(\alpha-1)q_{+}>-1$. Also, $p_{+}<\infty\implies q_{-}>1$, therefore both integrals above are finite and $\widehat{g}\in L^{\q}(\R)$. 
\end{example}

\begin{theorem}[Fourier inversion in norm]\label{theo_inversion}
     Consider $\p\in\mathcal{P}(\R)$ such that $p_{+}<\infty$ or $1/\p$ is $\LH$, and $f\in \Lp$. Let $\psi\in L^1(\R)$ be such that $\int_{\infty}^\infty\psi(t)dt=1$, $\widehat{\psi}$ is absolutely continuous, $\int_{-\infty}^\infty|s|^{1/p_{-}}|\widehat{\psi}(s)|ds<\infty$ and $\int_{-\infty}^\infty|s|^{1/p_{-}}|\widehat{\psi}'(s)|ds<\infty$. For each $a>0$, define $\psi_a(x)=\psi(x/a)/a$. Consider the family of kernels $K_a(s)=\widehat{\psi_a}(s)$. Also, let $e_x(s)\defeq e^{ixs}$ and $I_a[f](x)=(2\pi)^{-1}\int_{-\infty}^\infty e_x K_a\widehat{f}$. Then 
    \begin{equation*}
        \lim_{a\to 0^{+}}\|f-I_a[f]\|_{\p}=0.
    \end{equation*}
\end{theorem}
\begin{proof}
    First notice that since $\widehat{\psi}$ is absolutely continuous, we have that $\widehat{\psi}\in\BV(\R)$. Also, by the same argument as in the proof of Theorem \ref{theo_exchange}, we have that $\widehat{\psi}(s)=o(|s|^{1/p_{-}})$. Therefore, by a proof similar to that of Theorem \ref{theo_exchange}, we have that the exchange formula applies to the function $g(s)= e^{ixs}K_a(s)$. Notice also that by our assumptions we have that $\widehat{\psi}\in L^1(\R)$.  Then 
    \begin{align*}
        I_a[f](x)&=(2\pi)^{-1}\int_{-\infty}^{\infty}f(t)\widehat{K_a}(t-x)dt\\
        &=\int_{-\infty}^{\infty}f(t)\psi_a(x-t)dt\\
        &=f*\psi_a(x).
    \end{align*}
    Theorem 5.4 in \cite{Cruz-Uribe} then implies that 
    \begin{equation*}
        \lim_{a\to 0^+}\|f-I_a[f]\|_{\p}=0,
    \end{equation*}
    as claimed.
\end{proof}
\begin{obs}
    As mentioned in \cite{Talvila2025}, the following three commonly used kernels satisfy the conditions from Theorem \ref{theo_inversion}:\\
    The Cesàro-Fejér kernel
    \begin{equation*}
        K_a(s)=(2\pi)^{-1}(1-a|s|)\chi_{[-1/a,1/a]}(s),
    \end{equation*}
    the Abel-Poisson kernel
    \begin{equation*}
        K_a(s)=\pi^{-1}e^{-a|s|},
    \end{equation*}
    and the Gauss-Weirstrass kernel
    \begin{equation*}
        K_a(s)=(2\pi)^{-1}e^{-a^2s^2}.
    \end{equation*}
\end{obs}

\begin{theorem}
    Let $f\in\Lp$, with $\p\in\mathcal{P}(\R)$ such that $p_{+}<\infty$ or $1/\p$ is $\LH$. Suppose $g_1\in L^1(\R)\cap\BV(\R)$ is such that $\int_{-\infty}^\infty|s|^{1/p_{-}}|dg_1(s)|<\infty$. Suppose also that $g_2\in L^1(\R)$ satisfies  $\int_{-\infty}^\infty|s|^{1/p_{-}}|g_2(s)|ds<\infty$. Then 
    \begin{equation}
        \int_{-\infty}^\infty \widehat{f}g_1*g_2=\int_{-\infty}^\infty f(s)\widehat{g_1}(s)\widehat{g_2}(s)ds.
    \end{equation}
    \end{theorem}
    \begin{proof}
        The proof follows the same arguments as the the proof of Theorem 7.4 in \cite{Talvila2025}.
    \end{proof}

\section*{Acknowledgements}

The authors would like to thank professor Alexandre Kirilov for his assistance, opinions and insights during the development of this paper.


\begin{thebibliography}{99}

\bibitem{Chamorro} Chamorro, D. and  Vergara-Hermosilla, G. {Lebesgue spaces with variable exponent: some applications to the Navier–Stokes equations},  {\it Positivity} \textbf{28} (2024), no. 24.

\bibitem{Cruz-Uribe-Stieltjes}
Cruz-Uribe, D. V. and Convertito, G.
\textit{The Stieltjes Integral}, 1st ed., Chapman and Hall/CRC, 2023.

\bibitem{Cruz-Uribe} 
Cruz-Uribe, D. V. and Fiorenza, A.  
\textit{Variable Lebesgue Spaces: Foundations and Harmonic Analysis}, Birkhäuser/Springer, Heidelberg, 2013.

\bibitem{CUHM16}
Cruz-Uribe, D., Hernández, E., and Martell, J. M.
Greedy bases in variable Lebesgue spaces, 
\textit{Monatsh. Math.} \textbf{179}(3) (2016), 355–378. 



\bibitem{Diening2004}
Diening, L. 
Maximal function on generalized Lebesgue spaces $L^{p(\cdot)}$,
\textit{Math. Inequal. Appl.} \textbf{7} (2004), no. 2, 245–253.

\bibitem{Diening2004b}
Diening, L. 
Riesz potential and Sobolev embeddings of generalized Lebesgue and Sobolev spaces $L^{p(\cdot)}$ and $W^{k,p(\cdot)}$,
\textit{Math. Nachr.} \textbf{263} (2004), no. 1, 31–43.

\bibitem{key-estimate}
Diening, L. and Schwarzacher, S. 
On the Key Estimate for Variable Exponent Spaces,
\textit{Azerbaijan J. Math.} \textbf{3} (2013), no. 2, 62-69.


\bibitem{Hasto2009} 
Hästö, P.  
Local-to-global results in variable exponent spaces, 
\textit{Math. Res. Lett.} \textbf{16} (2009), no. 2, 263–278.

\bibitem{KovacikRakosnik1991} 
Kováčik, O. and Rákosník, J.
On spaces $L^{p(x)}$ and $W^{k,p(x)}$, 
\textit{Czech. Math. J.} \textbf{41}(116) (1991), 592–618.

\bibitem{Orlicz1932} 
Orlicz, W.
Über eine gewisse Klasse von Räumen vom Typus B, 
\textit{Bull. Int. Acad. Pol. Sci. Ser. A} \textbf{8} (1932), 207–220.

\bibitem{Talvila2007}
Talvila, E. 
The distributional Denjoy integral, 
\textit{Real Anal. Exchange} \textbf{33} (2007), no. 1, 51–84.

\bibitem{Talvila2025}
Talvila, E. 
The Fourier Transform in Lebesgue spaces, 
\textit{Czech. Math. J.} \textbf{75} (2025), 179–191.

\bibitem{Sam98}
Samko, S. G.
Convolution and potential type operators in $L^{p(x)}(\mathbb{R}^n)$, 
\textit{Integral Transforms and Special Functions} \textbf{7}(3--4) (1998), 261--284. 

\bibitem{heat_equation}
Vergara-Hermosilla, G. 
On variable Lebesgue spaces and generalized nonlinear heat equations, 
\textit{Preprint}, arXiv:2404.09588, 2024.

\bibitem{log-2019}
Yang, S., Yang, D. and Yuan, W. 
New characterizations of Musielak-Orlicz-Sobolev spaces via sharp ball averaging functions,
\textit{Front. Math. China} \textbf{14} (2019), 177–201.

\end{thebibliography}
\end{document}